\documentclass[11pt, reqno]{amsart} 
\usepackage{style}

\newcommand{\ba}{\mathbf{a}}
\newcommand{\bb}{\mathbf{b}}

\newcommand{\be}{\mathbf{e}}

\newcommand{\bp}{\mathbf{p}}
\newcommand{\bs}{\mathbf{s}}

\newcommand{\bu}{\mathbf{u}}
\newcommand{\bv}{\mathbf{v}}
\newcommand{\bw}{\mathbf{w}}
\newcommand{\bx}{\mathbf{x}}
\newcommand{\by}{\mathbf{y}}
\newcommand{\bz}{\mathbf{z}}

\newcommand{\bB}{\mathbf{B}}
\newcommand{\bA}{\mathbf{A}}

\newcommand{\0}{\mathbf 0}

\newcommand{\R}{\mathbb{R}}
\newcommand{\Z}{\mathbb{Z}}
\newcommand{\N}{\mathbb{N}}
\newcommand{\Q}{\mathbb{Q}}

\newcommand{\Scal}{\mathcal{S}}
\newcommand{\Tcal}{\mathcal{T}}

\DeclareMathOperator{\conv}{conv}
\DeclareMathOperator{\aff}{aff}
\DeclareMathOperator{\vol}{vol}

\DeclareMathOperator{\ldiam}{diam_\Z}
\DeclareMathOperator{\diam}{diam}

\newcommand{\size}{\operatorname{size}}
\newcommand{\NPclass}{\mathsf{NP}}
\newcommand{\Pclass}{\mathsf{P}}
\newcommand{\RPclass}{\mathsf{RP}}
\newcommand{\SVP}{\operatorname{SVP}}

\DeclareMathOperator{\blocks}{blocks}

\DeclareMathOperator{\LD}{LD}
\DeclareMathOperator{\Ehr}{L}

\DeclareMathOperator{\BG}{\mathsf{BG}} 
\DeclareMathOperator{\NP}{\mathsf{NP}} 

\title{\scshape On Lattice Diameter Segments: \mbox{Algorithms and Structure}}
\date{}

\author{G.~Averkov, A.~E.~Brose, J.~A.~De~Loera, G.~Lopez-Campos, A.~J.~Torres}

\begin{document}
\begin{abstract}
    \noindent The lattice diameter of a bounded set $S \subset \R^d$ counts the maximal number of lattice points in a segment whose endpoints are lattice points in $S$. Such a segment is called a lattice diameter segment of $S$. This basic invariant in the geometry of numbers has several interesting applications. 
    Our first type of contribution is algorithmic:
    We describe a polynomial-time algorithm that computes lattice diameter segments of polytopes in fixed dimension, with a simpler algorithm for polygons. In contrast, we show hardness of computing the lattice diameters of polytopes of variable dimension and for semi-algebraic sets. Our second type of contribution concerns the structure of lattice diameter segments:
   Generalizing the classical Ehrhart function, we prove that the function that counts distinct lattice diameter segments in dilations of a lattice polytope is eventually a quasi-polynomial in the dilation factor. Finally, we prove a Borsuk-type theorem on the number of parts needed to partition a set of lattice points such that each part has strictly smaller lattice diameter.

\end{abstract}

\maketitle

\mscclasses{\emph{Mathematics Subject Classification.}
52A38, 
52B20, 
52B55, 
52C07, 
52C17, 
68Q17. 
}


\section{Introduction}  

We study a fundamental geometric invariant of compact sets, called the \textit{lattice diameter}, with a variety of applications in discrete geometry, integer optimization, and the geometry of numbers.  Here are the key definitions of this paper:

\begin{definition}\label{def:ld-non-convex}
    For a compact set $S \subset \R^d$, and $\bu \in \Z^d \setminus \{\mathbf{0}\}$ define the \textit{lattice diameter in direction $\bu$} as
    \[\ldiam(S, \bu):=\max_{\substack{\bx,\by \in S \cap \Z^d \\ \bx - \by \in \Z\bu }}
    |[\bx,\by] \cap \Z^d|-1\]
    where $[\bx,\by]$ denotes the segment from $\bx$ to $\by$. Then define the \textit{lattice diameter} as
    \[
    \ldiam(S) := \max_{\bu \in \Z^d \setminus \{\0\}} \ldiam(S,\bu).
    \]
    A direction $\bu \in \Z^d \setminus \{\0\}$ maximizing $\ldiam(S)$ is called a \textit{lattice diameter direction}, a segment $[\bx,\by]$ maximizing $\ldiam(S, \bu)$ is called a \emph{lattice diameter segment}, and the corresponding line through $\bx$ and $\by$ is called a \emph{lattice diameter line}. 
\end{definition}

\emph{Our paper investigates properties of the lattice diameter, lattice diameter segments, and algorithms to compute them.}

{\bf Motivation and related work.} It is widely known that optimization algorithms benefit from the understanding of concrete geometric measurements. For example, the running time of the Ellipsoid method is stated using
the \emph{radii} of a largest inscribed sphere and a smallest enclosing sphere (see \cite{GLS1993} for details on this and many other algorithms where geometric measures matter). Our work proposes the lattice diameter as a useful measurement for computations with lattice points. Next we present some applications and connections.

A famous geometric measurement in combinatorial optimization is the \emph{lattice width} $w(P)$. It is used in many algorithms for integer programming such as Lenstra's algorithm \cite{ILP_poly,GLS1993}. It turns out that the lattice width and the lattice diameter bound each other up to a constant. In dimension two the lattice width is tightly bounded by the lattice diameter up to a constant $4/3$ (see \cite{Barany_Furedi}). More generally, one can prove, using an estimate of lattice width in \cite[Proposition~2.2]{AverkovetAl2026ExactFlatness} and a bound by Rabinowitz \cite{rabi89}, that for any non-empty lattice polytope $P$ there is a constant $C_d$, only depending on the dimension, such that 
\[
\quad w(P)\le C_d(\ldiam(P)+1).\quad
\]
This constant $C_d$ is known to be at most $\operatorname{Flt}(d)$, but it is far from tight.

In discrete tomography, Gardner and Gritzmann have studied the reconstruction of a set of lattice points from its ``$X$-rays'', i.e., the collection of line sums over parallel lattice lines~\cite{gardnergritzmann97}.
Note that lattice segments are the largest integer one-dimensional sections of a body.
Motivated by questions in analysis and number theory, the volumes of sections of convex bodies have been studied for many years; see the surveys \cite{giannopoulos2023inequalities,nayar2023extremal}. In particular, for the challenge of determining a maximal volume hyperplane section of a convex body there are now efficient algorithms \cite{BDLM}. In comparison, the analogous problem for maximizing the number of lattice points inside a section has only been studied more recently \cite{gard05,freyerhenk22,freyerhenk24}. The computational complexity of finding a $k$-dimensional section of a polytope containing the most lattice points remains wide open. Here we solve the complexity of one-dimensional lattice sections, i.e., lattice segments.

Furthermore, classifications of convex lattice polytopes have been studied with respect to different invariants. This is of particular interest in algebraic geometry, where reflexive polytopes correspond to specific toric Fano varieties that have ties to mirror symmetry in physics~\cite{batyrev1999classification,kasprzyk2006toric}. Arnold, B\'ar\'any, Pach, Vershik, and others~\cite{arno80,bara92,bara92a} have classified lattice polytopes by their volume while Averkov, Lagarias, Nill, Weismantel, Ziegler, and others \cite{lagariasZiegler1991bounds, averkov2011maximal, nillziegler2011projecting} have continued the classification, based on the number of interior lattice points.
We stress that classifying lattice polytopes by their lattice diameter is also possible. A simple old result by S. Rabinowitz \cite{rabi89} shows that if $\ldiam(P)<m$, then $|P \cap \Z^d| \leq m^d$.
Also in more recent work, Dillon and Arun~\cite{HollowP} asymptotically bound the number of vertices that a lattice polytope with fixed lattice diameter can have.

Another motivation for our paper is Ehrhart theory. In its simplest version, for a $d$-dimensional lattice polytope $P$, and $k \in \N$, the Ehrhart counting function, $\Ehr_P(k) = |kP \cap \Z^d|$, agrees with a polynomial in $k$ of degree $d$. This foundational result extends to rational, multivariate, and weighted versions, using generating functions and featuring reciprocity theorems \cite{beckrobins}. Here we count not individual lattice points, but subsets of them.
It can be shown that there is an Ehrhart theorem for counting lattice lines with fixed direction in a rational polytope, that pass through a lattice point of the rational polytope \cite{aebThesis}.
In our paper we study the counting function of lattice diameter lines (or segments) of a lattice polytope.

The final motivation for our work is the famous Borsuk's partition problem. In 1933, Borsuk posed the following question in \cite{borsuk1933drei}:
\begin{quote}
    \emph{Can every bounded set $S \subset \R^d$ be partitioned into $d+1$ sets with \mbox{smaller diameter?}}
\end{quote}
Here Borsuk refers to the (continuous) Euclidean diameter.
The answer to Borsuk's question is positive for $d=2$, see \cite{borsuk1933drei}, and for $d=3$, see \cite{perkal1947subdivision}, \cite{eggleston1955covering} and \cite{grunbaum1957simple}. In higher dimensions it was believed to be true until Kahn and Kalai \cite{kahn1993counterexample} constructed a counterexample for $d=1325$ and for each $d>2014$. Today, there are known counterexamples in dimensions $63$ and higher, so the problem is still open for $4\leq d\leq 62$. We want to note that the counterexample in dimension $63$ was only recently found by Grinsztajn with help of AI \cite{Grinsztajn2026}.

More generally, the problem of partitioning a set into sets with smaller diameter is studied for specific classes of bounded sets. In this paper, we study this problem for sets of lattice points with respect to the lattice diameter instead of the Euclidean diameter. We call this challenge the \textit{discrete Borsuk partition problem}.

\subsection*{Our Contributions}

In \Cref{sec:LD-poly} we compute lattice diameters and their segments. It is known that when the dimension is fixed, then the number of lattice points in a rational polytope $P$ \cite{latte1,barvinok1994polytime}, as well as the lattice width of a rational polytope \cite{Sebo_latticewidth} can be computed in polynomial time, whereas when the dimension is variable, computing these invariants is $\#\Pclass$-hard, and $\NPclass$-hard, respectively. We show the same phenomenon for lattice diameters.

\begingroup \hypersetup{hidelinks} \begin{manualtheorem}{\Cref{thm:LD-poly-anydimension}}
    Let the dimension $d \in \Z_{\geq 1}$ be fixed. Then a lattice diameter segment of a rational polytope $P \subset \R^d$ can be computed in polynomial time.
\end{manualtheorem} \endgroup 

In the plane, we have a better understanding of lattice diameter segments and present a faster algorithm that additionally computes all lattice diameter directions. We describe its complexity in the arithmetic complexity model, where arithmetic operations ($+,-,\cdot, \div $) and comparisons are unit-cost operations, see \cite{Schrijver}.

\begingroup \hypersetup{hidelinks} \begin{manualtheorem}{\Cref{thm:LD-poly-dim2}}
     Let $P$ be a lattice polygon with $n$ vertices whose coordinates have binary encoding length at most $c$. \Cref{alg:2d} computes all lattice diameter directions of $P$, and at least one representative lattice diameter segment per direction. The algorithm uses 
    $O(n \log n + nc)$ arithmetic operations on rational numbers of size $O(c)$. If the vertices are given in cyclic order, then the number of arithmetic operations is $O(nc)$.
\end{manualtheorem} \endgroup

In contrast, we show hardness of computing the lattice diameters in more general settings in \Cref{sec:NP-hard}. First, we prove that in non-fixed dimension computing the lattice diameter is at least as hard as a shortest vector problem ($\SVP$):

\begingroup \hypersetup{hidelinks} \begin{manualtheorem}{\Cref{thm:NP-hard-crosspolytope}}
	There exists a polynomial-time reduction from a factor-two approximation of $\SVP$ with respect to the $\ell_1$-norm for a given lattice to computation of the lattice diameter of a given $0$-symmetric lattice cross-polytope. In particular, it is hard to approximate the lattice diameter of a $0$-symmetric lattice cross-polytope within a constant factor unless $\NPclass \subseteq \RPclass$.
\end{manualtheorem} \endgroup

The conclusion of \Cref{thm:NP-hard-crosspolytope} relies on the hardness result by Regev and Rosen \cite[Corollary~4.3]{Regev_Rosen} on approximating $\SVP$ within a constant factor. The condition that a problem is \textit{hard to approximate unless $\NPclass \subseteq \RPclass$} means that it is $\NPclass$-hard under randomized reduction. See \cite{AroraBarak} for more details.

Moreover, in fixed dimension, but for semi-algebraic sets, we have:

\begingroup \hypersetup{hidelinks} \begin{manualtheorem}{\Cref{thm:NP-hard-semialgebraic-set}}
    It is $\NP$-hard to compute the lattice diameter of a two-dimensional compact basic closed semi-algebraic set $K \subset \R^2$ that satisfies $K \cap \Z^2 \neq \emptyset$. It stays $\NP$-hard even to compute the lattice diameter in the fixed direction $(0,1)$.
\end{manualtheorem} \endgroup

This $\NP$-hardness result can be lifted to any fixed dimension $d \geq 2$.

In \Cref{sec:QP} we generalize Ehrhart theory to counting the set of lattice diameter lines (or segments). While counting the number of all lattice lines in a fixed direction that intersect a polytope $P$ in a lattice point can be easily done \cite{aebThesis}, counting just those that are lattice diameter lines is much more difficult. Surprisingly, we have an Ehrhart-type result.

\begin{definition}\label{def:LD_P}
   For a lattice polytope $P$ and a positive integer $k$, let $\LD_P(k)$ denote the number of lattice diameter segments of the dilated polytope $kP$.
\end{definition}

Then we show the following:

\begingroup
\hypersetup{hidelinks}
\begin{manualtheorem}{\Cref{thm:QP-all-dimensions}} Let $P \subset \R^d$ be a $d$-dimensional lattice polytope. Then the counting function $\LD_P(k):\N \to \N$ eventually agrees with a quasi-polynomial in the dilation variable.
\end{manualtheorem}
\endgroup

In \Cref{sec:Borsuk} we study the discrete version of Borsuk's partition problem.

\begin{definition}
Let $S \subset \Z^d$ be a bounded set. Define the \textit{lattice Borsuk number} of $S$ as the minimal size of a partition of $S$ where each part has a smaller lattice diameter than $\ldiam(S)$ and denote it $\beta_\Z(S)$.
\end{definition} 

The following theorem gives a complete answer to the discrete version of Borsuk's partition problem:

\begingroup
\hypersetup{hidelinks}
\begin{manualtheorem}{\Cref{thm:Borsuk}}
    Let $S \subset \Z^d$ be a bounded set. Then $\beta_\Z(S) \leq 2^d$ and this bound is best possible.
\end{manualtheorem}
\endgroup

In particular, \Cref{thm:Borsuk} says that $\beta_\Z(P \cap \Z^d) \leq 2^d$ for $d$-dimensional polytopes $P$. Note that we are not partitioning $P$ into smaller polytopes, i.e., the parts of $P \cap \Z^d$ are not necessarily the lattice points of a polytope. The bound $2^d$ is also best possible for the set of lattice points in a polytope; it is attained, for example, for the lattice points of a $d$-dimensional cube.

{\bf Remarks on the definition and related notions.} We are not the first to study lattice diameters; this notion was introduced by Corzatt and Stolarsky~\cite{Corzatt_PhD}, studied by Alarcon~\cite{Alarcon_PhD}, and later by Bárány and Füredi~\cite{Barany_Furedi}. However, they focused mainly on convex sets. Note that in the more general setting, when $S$ is non-convex, the lattice diameter of $S$ and $\conv(S)$ can differ. When $S$ is a convex body, then a lattice diameter line is an affine line $L$ such that $|L \cap S \cap \Z^d|$ is maximal among all affine lines.

There are two related notions of diameter in the literature: 
The continuous Euclidean \textit{diameter} of a compact body $S$, denoted $\diam (S)$, is
simply the supremum of the Euclidean distances between pairs of points in $S$. It has been used in the performance of the center of gravity algorithm for convex optimization \cite{bubeck2015convex, levin1965algorithm}. Note that for fixed lattice diameter, the Euclidean diameter can be arbitrarily large (\Cref{fig:unbounded_ed}). Moreover, Euclidean diameter lines of polytopes always pass through vertices \cite[Lemma~1.7]{GritzmannKlee1992}, while lattice diameter lines \mbox{may not (\Cref{fig:interior_2d}).}

\begin{figure}[ht!]
  \centering
  \begin{subfigure}[t]{0.45\textwidth}
    \centering
    \includegraphics[width=\linewidth]{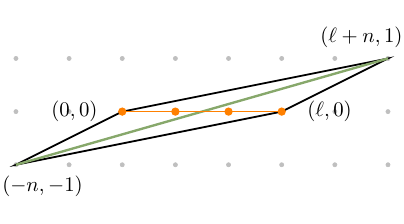}
    \caption{Lattice polygons with fixed lattice diameter \mbox{(orange)} and arbitrarily large Euclidean diameter (green).}
    \label{fig:unbounded_ed}
  \end{subfigure}\hfill
  \begin{subfigure}[t]{0.45\textwidth}
    \centering
    \includegraphics[width=\linewidth]{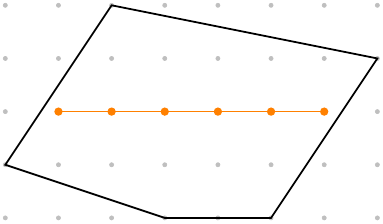}
    \caption{A lattice diameter segment contained in the \mbox{interior} of a lattice polygon.}
    \label{fig:interior_2d}
  \end{subfigure}
  \caption{Examples of lattice diameter segments in lattice polygons.}
  \label{fig:two_subfigs}
\end{figure}

Another related notion is the \textit{rational diameter}, which we denote by $\diam_\Q$. It measures the longest rational segment in $P$ with respect to its lattice normalized length. The rational diameter has  been studied \cite{CodenottiFreyer24,soprunov16} and has applications in, for example, coding theory.
For rational polytopes $P$ one can show that $\ldiam(P) \leq \lfloor \diam_\Q(P) \rfloor \leq \diam (P)$, where equality holds in dimension two (see \Cref{lemma:edge_opposite_vertex}).

\subsection*{Notation}
We use basic notions and notation as can be found in the textbooks~\cite{grub07,zieglerbook}.
A $d$-polytope $P$ is the convex hull of finitely many points in $\R^d$ whose affine hull is $d$-dimensional. It is called a \emph{lattice} (or \emph{rational}) $d$-polytope if its vertices belong to the integer lattice $\Z^d$ (or to $\Q^d)$. A polytope $P$ is called \textit{$0$-symmetric} if $P = -P$. We say two polytopes $P,Q \subset \R^d$ are unimodularly equivalent if there exists an affine map $f$ with $f(P)=Q$ and $f(\Z^d)=\Z^d$. We write $\aff(S)$ and $\conv(S)$ for the affine and convex hulls of a set $S \subset \R^d$, respectively.
A \emph{line} will always mean an \emph{affine line} and a \emph{lattice line} is a line that passes through at least two lattice points. Similarly, a segment $[\bx,\by]$ is called a lattice segment if it contains at least two lattice points. We say that a line has direction $\bu$ if it can be written as $\bx+\R\bu$ for some $\bx \in \R^d$. We use the standard inner product on $\R^d$, that is, $\langle \bx, \by \rangle := \sum_{i=1}^{d} x_iy_i$ for $\bx,\by \in \R^d$.  We write $\vol$ for the Euclidean volume of a Lebesgue measurable set in $\R^d$, and $\vol_1$ for the Euclidean length of segments in $\R^d$. Furthermore, a vector $\bu=(u_1, \ldots, u_d) \in \Z^d$ is called primitive if $\gcd(u_1, \ldots, u_d) =1$ and we use the convention $\gcd(0, \ldots, 0)=0$.


\section{A polynomial-time algorithm for the lattice diameter of rational polytopes in fixed dimension}\label{sec:LD-poly}

	It is natural to try and understand if the lattice diameter can be computed efficiently. It turns out that this is indeed the case when the dimension of the ambient space is fixed.
	There are two degenerate situations we will handle separately. When $P \cap \Z^d$ is empty, we take the maximum over the empty set so that $\ldiam(P) = -\infty$. When $P \cap \Z^d$ consists of only one point, we have $\ldiam(P) = 0$. Otherwise, when $P \cap \Z^d$ consists of at least two points, we can rewrite the formula for the lattice diameter in the form 
	\[
		\ldiam(P) := \max \{ \gcd(\by-\bx) \colon \bx,\by \in P \cap \Z^d\}.
	\]

	\begin{theorem} \label{thm:LD-poly-anydimension} 
		Let the dimension $d \in \Z_{\ge 1}$ be fixed. Then a lattice diameter segment of a rational polytope $P \subset \R^d$ can be computed in polynomial time. 
	\end{theorem} 
    
	\begin{proof} 
		Barvinok showed that the number $N = | P \cap \Z^d|$ of lattice points in $P$ can be determined in polynomial time when the dimension is fixed \cite{barvinok1994polytime}. For $N=0$ and $N=1$ we have $\ldiam(P) = - \infty$ and $\ldiam(P)=0$, respectively. (The case $N=0$ can also be detected by solving a feasibility problem of integer linear programming.) We are left with the non-degenerate case $N \ge 2$.  In this case, the formula for the lattice diameter can be reformulated with the greatest common divisor and then further reformulated as 
		\begin{equation} \label{eq:ldiam-w-steps} 
			\ldiam(P) =  \max \{ s \in \Z_{\ge 1} \colon \by - \bx = s \cdot \bu, \ \bx, \by \in P \cap \Z^d, \bu \in \Z^d \setminus \{\mathbf{0}\} \} .
		\end{equation} 
		Here, $\bu$ encodes the step that we make on the way from $\bx$ to $\by$, and $s$ encodes how many such steps we make. 
		The right-hand side of \eqref{eq:ldiam-w-steps} is a non-linear integer optimization problem in $(\bx, \by, \bu,s) \in \Z^{3d+1}$. 
		Two sources of non-linearity are the quadratic terms $s \cdot \bu = (su_1,\ldots, s u_d)$ in the constraint $\by - \bx = s\cdot \bu$ and the condition $\bu \ne \mathbf{0}$ in $\Z^d \setminus \{\mathbf{0}\}$, which is a disjunction of the non-equality constraints $u_i \ne 0$. 
		
        For the non-equality constraint, note that exchanging $\bx$ and $\by$ and replacing $\bu$ with $-\bu$, the constraint $\by - \bx = s \cdot \bu$ remains satisfied. In view of this, we can strengthen $\bu \ne \mathbf{0}$ to $\bigvee_{i=1}^d (u_i \ge 1)$ without changing the maximum. With this modification, our formula for the lattice diameter gets regularized to
		\[
			\diam_\Z(P) =  \max \Bigl\{ s \in \Z_{\ge 1} \colon \by-\bx = s\cdot \bu, \ \bx, \by \in P \cap \Z^d, \ \bu \in \Z^d, \  \bigvee_{i=1}^d (u_i \ge 1) \Bigr\}.
		\]

		We still need to take care of the quadratic term $s\cdot \bu$. Every quantity that can be computed in polynomial time necessarily has polynomial description size. Hence, Barvinok's polynomial time algorithm for computing $N = | P \cap \Z^d|$ implies that $\log_2 N \le f ( \size(P))$, where $\size(P)$ is the description size of $P$
        and $f$ is  a polynomial. 
		
		As $1 \le s \le N-1$, we have a finite search space $\{s^-,\ldots,s^+\}$ with $s^-:=1$ and $s^+:=N-1$ of the possible values of $\ldiam(P)$. Consequently, for the determination of $\ldiam(P)$ we can do a binary search. During the search we maintain an integer interval $\{s^-,\ldots,s^+\}$ in which $\ldiam(P)$ lies, pick $s$ in the middle of the interval by rounding $(s^-+s^+)/2$ and solve $d$ feasibility problems  of integer linear programming (one for each $i \in [d]$) in order to update $s^-$ or $s^+$ and halve the size of the interval. The explicit formulation of the feasibility problem for a given $s$ and $i \in [d]$ is as follows: 
		\begin{equation} \label{eq:feasibility-problem}
        \begin{aligned}
        & \text{Decide if there exist $\bx,\by,\bu \in \Z^d$ satisfying:} \\
        & \by - \bx = s \cdot \bu, \quad
      \bA\bx \le \bb, \quad
      \bA\by \le \bb, \quad
      u_i \ge 1.
    \end{aligned}
    \end{equation}
		where $A\bx \le b$ is an inequality description of $P$. 
		The problem has $3d$ variables $(\bx, \by, \bu) \in \Z^{3d}$. Since $s$ has been fixed, the constraint $\by - \bx= s \cdot \bu$ is a system of $d$ linear equations. If, for a given $s$, at least one of the problems \eqref{eq:feasibility-problem} with $i \in [d]$ is feasible, we know that $\ldiam(P) \ge s$ so that we can update $s^-$ to $s^-:=s$. If all the problems \eqref{eq:feasibility-problem} with $i \in [d]$ are infeasible, we know that $\ldiam(P) < s$ and we can update $s^+$ to $s^+ := s-1$.
		
	Lenstra showed that integer linear programming is solvable in polynomial time if the dimension is fixed \cite{ILP_poly}. Thus, each feasibility problem \eqref{eq:feasibility-problem} solved in our binary search is solvable in polynomial time. In particular, we get a solution  $\bx,\by \in P \cap \Z^d$ and $\bu \in \Z^d$ satisfying $\by-\bx = \ldiam(P)\cdot \bu$ which corresponds to a lattice diameter segment.
		
		The above approach is a polynomial time algorithm for the determination of $\ldiam(P)$, which is based on solving  about $d \cdot  f(\size(P))$ feasibility problems of integer programming in the fixed dimension $3d$. The number of constraints in these problems is $2 m + d+1$, where $m$ is the number of inequalities in the system $\bA \bx \le \bb$ that describes $P$. 
	\end{proof} 

	\begin{remark} 
			The driving force of the algorithm from the proof of \Cref{thm:LD-poly-anydimension} is integer linear programming. What we need from Barvinok's algorithm is the fact, obtained as a byproduct, that the description size of $N$, and hence also that of $\ldiam(P)$, is polynomially bounded by the description size of $P$. We actually can adjust the initialization stage of our algorithm to avoid computing $N$. In the initialization stage we may first try to find a reasonably small upper bound on $\ldiam(P)$ quickly. For $N \geq 1$ this can be done as follows. 
			
			Iterate over $t=0,1,2,\ldots$, and solve \eqref{eq:feasibility-problem}  with $s=2^t$ for all  $i=1,\ldots,d$ in order to determine if $\ldiam(P) < 2^t$. As soon as we find $s$ such that, for each $i=1,\ldots,d$ the problem \eqref{eq:feasibility-problem} is infeasible, we have $2^{t-1} \leq \ldiam(P) < 2^t$. The initialization stage is finished and we can set $s^- = 2^{t-1}$ and $s^+ =2^t$. 
	\end{remark} 

	\begin{remark}
			If $P$ is a lattice polytope rather than a rational polytope, there is an alternative way to see that the description size of $\ldiam(P)$ and $N = |P \cap \Z^d|$ is polynomially bounded in the description size of $P$. Assume, without loss of generality, that $\dim(P) = d$. By Blichfeldt's inequality from the geometry of numbers \cite[Chapter~2]{gruber-lekkerkerker1987geometry}, one has $\vol(P) \ge \frac{1}{d!} (N-d)$. Consequently, $N \le d! \vol(P) + d$. Thus, for polynomially bounding the description size of $N$, it is enough to show that the description size of $\vol(P)$ is polynomially bounded. 
			To see this, pass to the vertex description of $P$ and calculate $\vol(P)$ using triangulation on the vertex set of $P$. In this way, $\vol(P)$ is expressed as a sum of polynomially many determinants. The description size of each determinant is polynomially bounded by standard arguments, see \cite[Section~3.2]{Schrijver}. 
	\end{remark} 

\subsection{A faster algorithm to compute lattice diameters in the plane}\label{sec:LD-poly-dim2}

We have seen that assuming fixed dimension allows us to efficiently compute the lattice diameter of any rational polytope. In this section, we give more insights on lattice diameters of lattice polygons. We present a faster algorithm that additionally computes all lattice diameter directions (in contrast to the general integer linear program described in \Cref{thm:LD-poly-anydimension}). 

We will also show that in dimension two the lattice diameter can be expressed and computed in terms of other invariants studied in the geometry of numbers. The \textit{rational diameter} of a convex body $C$ is 
\begin{align*}
\diam_\Q(C) &:= \max \{ s \in \R \colon \by - \bx = s \cdot \bu, \ \bx, \by \in C , \bu \in \Z^d \setminus \{ \mathbf{0}\} \} \\
&= \max_{\bu \in \Z^d \setminus \{ \mathbf{0}\}} \max_{\substack{\bx,\by \in C \\ \by-\bx \in \R \bu}} \frac{\vol_1([\bx,\by])}{\|\bu\|}.
\end{align*}
It measures the largest lattice normalized length of a segment contained in $C$. The \textit{first successive minimum} of a $0$-symmetric convex body $C$ is
\begin{align*}
    \lambda_1(C) := \min \{ \lambda > 0 \colon \lambda C \cap \Z^d\setminus \{ \mathbf{0} \} \neq \emptyset \},
\end{align*}
so $\diam_\Q(C) = \frac{1}{\lambda_1(C-C)}$. In \Cref{lemma:edge_opposite_vertex} we show that for lattice polygons $\ldiam(P)= \lfloor \diam_\Q(P) \rfloor$. 

We formulate our algorithm in terms of $P$ directly, but in \Cref{rem:alg-in-P-P} we explain how it can be viewed as computing $\lambda_1(P-P)$. We proceed to prove some lemmas. First, we give a tight bound between the lattice normalized length of a segment in $P$ and its lattice point count.

\begin{lemma}\label[lemma]{lemma:nvol}
    Let $P \subset \R^d$ be a lattice polytope and $L$ be a line in direction $\bu$, for primitive $\bu \in \Z^d$, with $L \cap P \cap \Z^d \neq \emptyset$. Then
    \begin{align*}
    \left \lfloor \frac{\vol_1(L \cap P)}{\| \bu\|} \right \rfloor \leq |L \cap P \cap \Z^d| \leq  \left \lfloor \frac{\vol_1(L \cap P)}{\| \bu\|} \right \rfloor + 1
    \end{align*}
and if an endpoint of $L\cap P$ is a lattice point, then the upper bound is attained with equality.
\end{lemma}

\begin{proof}
Parametrize the line as $L = \{\bx_0 + t\bu : t \in \mathbb{R}\}$, where $\bx_0$ is an arbitrary lattice point on $L$. Then, the intersection $L \cap P$ is a segment with endpoints $\ba = \bx_0 + t_1\bu$ and $\bb = \bx_0 + t_2\bu$, for some real numbers $t_1 \leq t_2$.
The lattice points in $L \cap P$ are precisely those points of the form $\bx_0 + k\bu$ with $k \in \mathbb{Z} \cap [t_1, t_2]$. Moreover,
\[
\frac{\vol_1(L \cap P)}{\|\bu\|} = \frac{\|\bb - \ba\|}{\|\bu\|} = \frac{\|(t_2 - t_1) \bu\|}{\|\bu\|} = |t_2 - t_1|=t_2-t_1.
\]

If neither $t_1$ nor $t_2$ is an integer, then $|L \cap P \cap \mathbb{Z}^d| 
= \lfloor t_2 \rfloor - \lceil t_1 \rceil + 1 = \lfloor t_2 \rfloor - \lfloor t_1 \rfloor$
and the claim follows from the observation
$\lfloor t_2-t_1 \rfloor \leq \lfloor t_2 \rfloor - \lfloor t_1 \rfloor \leq \lfloor t_2-t_1 \rfloor +1$. If at least one of $t_1$ or $t_2$ is an integer, then
$|L \cap P \cap \Z^d| = \lfloor t_2 \rfloor - \lceil t_1 \rceil + 1 = \lfloor t_2 - t_1 \rfloor + 1=\left \lfloor \frac{\vol_1(L \cap P)}{\| \bu\|} \right \rfloor + 1$.
\end{proof}

The following lemma describes where to find lattice diameter segments. When $\bu \in \Z^d \setminus \{ \mathbf{0}\}$ is fixed, we write the \textit{rational diameter in direction $\bu$} as $\diam_\Q(P,\bu)$, i.e., the maximal lattice normalized length of a segment in $P$ in direction $\bu$.
For an edge $e$ of $P$ with outward normal $\ba_e$ let $V(e)$ be the vertices of the face minimizing $\ba_e$ over $P$, so in dimension two $|V(e)| \in \{1,2\}$. We call $\bv \in V(e)$ an \textit{opposite vertex to $e$}.

\begin{lemma}\cite[Lemma 1]{Barany_Furedi}
\label[lemma]{lemma:edge_opposite_vertex}
    Let $P$ be a lattice polygon and $\bu\in \Z^2 \setminus \{\mathbf{0}\}$.
    There exists a rational diameter segment in direction $\bu$ with endpoints on an edge $e$ and an opposite vertex $\bv$. Furthermore, the rational diameter segment contains a lattice diameter segment in direction $\bu$ and $\ldiam(P,\bu) = \lfloor \diam_\Q(P, \bu) \rfloor$.
\end{lemma}

\begin{proof}
    This follows from \cite[Lemma~1]{Barany_Furedi} but is stated without a proof so we give a sketch here. First, note that in the plane, a longest segment (in Euclidean length) in direction $\bu$ can always be chosen to lie between an edge $e$ and an opposite vertex $\bv$. Hence, such a segment is a rational diameter segment in direction $\bu$. Since it contains a vertex, this segment also contains lattice points and attains equality in the upper bound of \Cref{lemma:nvol}. This proves the claim.
\end{proof}

In the following, we call a lattice segment starting at a vertex $\bv$ of a polytope $P$ and containing the most lattice points among all such segments in $P$ a \textit{local lattice diameter segment at $\bv$}. The next lemma describes how to find local lattice diameter segments.

\begin{lemma}\label[lemma]{lemma:min-gens}
    Let $T = \conv (e \cup \{ \bv\})$ be a lattice triangle, where $e$ is an edge of $T$ with outward normal vector $\ba$, and $\bv$ is its opposite vertex. If $\bw_1,\bw_2 \in T \cap \Z^2$ are such that $\bv-\bw_1$ and $\bv-\bw_2$ are primitive, and $\langle \ba, \bw_1 \rangle \leq \langle \ba, \bw_2 \rangle$, then for $L_1:= \aff(\{\bv,\bw_1\})$ and $L_2:= \aff(\{\bv,\bw_2\})$
    \[
    |L_1 \cap T \cap \Z^2| \geq |L_2 \cap T \cap \Z^2|.
    \]
\end{lemma}

\begin{proof}
    If $\langle \ba, \bw_1 \rangle = \langle \ba, \bw_2 \rangle$, then by  proportionality of lines 
\[
\frac{\operatorname{vol}(L_1 \cap T)}{\|\bw_1- \bv\|} = \frac{\operatorname{vol}(L_2 \cap T)}{\|\bw_2- \bv\|}.
\]
Applying \Cref{lemma:nvol}, we obtain
\[
|L_1 \cap T \cap \mathbb{Z}^2| = \left\lfloor \frac{\operatorname{vol}(L_1 \cap T)}{\|\bw_1- \bv\|} \right\rfloor + 1= \left\lfloor \frac{\operatorname{vol}(L_2 \cap T)}{\|\bw_2- \bv\|} \right\rfloor + 1 = |L_2 \cap T \cap \mathbb{Z}^2|.
\]
On the other hand, if $\langle \ba, \bw_1 \rangle < \langle \ba, \bw_2 \rangle$, then there exists a $\bw' \in [\bv, \bw_2]$ such that $\langle \ba, \bw' \rangle = \langle \ba, \bw_1 \rangle$. Again, by proportionality of lines
\[
\frac{\operatorname{vol}(L_1 \cap T)}{\|\bw_1- \bv\|} = \frac{\operatorname{vol}(L_2 \cap T)}{\|\bw'- \bv\|} \geq \frac{\operatorname{vol}(L_2 \cap T)}{\|\bw_2- \bv\|}.
\]
Applying \Cref{lemma:nvol} it follows that $|L_1 \cap T \cap \mathbb{Z}^2| \geq |L_2 \cap T \cap \mathbb{Z}^2|$.
\end{proof}

The lemmas above give rise to our algorithm:

\begin{algorithm}[ht!]
\caption{\textsc{Computing Lattice Diameter Directions and Segments}}\label{alg:2d}
\begin{algorithmic}[1]
\Require a lattice polygon $P$ given by its vertices ordered cyclically
\Ensure at least one lattice diameter segment for every lattice diameter direction
\State Initialize two empty sets: $\Tcal$ for triangles, and $\Scal$ for candidate segments
\ForAll{edges $e$ of $P$}
    \State compute its opposite vertex/vertices $V(e)$
    \State add a triangle $\conv(e \cup \{\bv\})$ to $\Tcal$ for every $\bv \in V(e)$
\EndFor
\ForAll{triangles $T = \conv(e \cup \{\bv\}) \in \Tcal$}
    \State compute up to three local lattice diameter segments of $T$ at $\bv$
    \State add these local lattice diameter segments to $\Scal$
\EndFor
\State \Return all $S \in \Scal$ that maximize $|S \cap P \cap \Z^2|$
\end{algorithmic}
\end{algorithm}

\begin{theorem}\label{thm:LD-poly-dim2}
    Let $P$ be a lattice polygon with $n$ vertices whose coordinates have binary encoding length at most $c$. \Cref{alg:2d} computes all lattice diameter directions of $P$, and at least one representative lattice diameter segment per direction. The algorithm uses 
    $O(n \log n + nc)$ arithmetic operations on rational numbers of size $O(c)$. If the vertices are given in cyclic order, then the number of arithmetic operations is $O(nc)$.
\end{theorem}

\begin{proof}   
    
    \emph{Correctness of the algorithm:} Let $\bu$ be a lattice diameter direction of $P$. By \Cref{lemma:edge_opposite_vertex} there exists a lattice diameter segment $S$ in direction $\bu$, between an edge $e$ and an opposite vertex $\bv$ to that edge. Thus $S$ is a local lattice diameter segment of $T=\conv(e \cup \{\bv\})$ at $\bv$. By \Cref{rem:at-most-three-lld} at most three lattice diameter segments pass through $\bv$, hence any lattice diameter segment is one of the local lattice diameter segments in $\Scal$.

    \emph{Complexity of the algorithm:}
    If $P$ is given by vertices in arbitrary order, then one can obtain a cyclic order of the vertices (and hence also the edges) in $O(n \log n)$ arithmetic operations. From now on we assume that the vertices (and edges) are given in cyclic order.
    
    \begin{enumerate}
        \item \emph{first for loop:} An opposite vertex $\bv$ to an edge $e$ with outward normal vector $\ba$ means that $\bv$ minimizes the linear function $\langle \ba, \cdot \rangle$ over $P$. An opposite vertex to some first edge can be found by exhaustive search in $O(n)$. Iterating through the further edges in cyclic order, and advancing vertices (in the cyclic order) until they are minimizers of the current edge, we will test each vertex a constant number of times. Hence, all opposite vertex to edge pairings can be obtained in $O(n)$. Each edge contributes at most two triangles, so throughout the whole algorithm $|\Tcal| \leq 2n$.
        \item \emph{second for loop:} Consider a fixed $T = \conv(e \cup \{\bv\})$ and let $\ba$ be a primitive outward normal vector of $e$. By \Cref{lemma:min-gens}, finding a local lattice diameter of $T$ at $\bv$ is equivalent to finding $\bw_1 \in T \cap \Z^2 \setminus \{\bv\}$ that minimizes $\langle \ba, \cdot \rangle$. (A segment spanned by $\bv$ and $\bw_1$ is a local lattice diameter segment.) 
        The lattice points $T \cap \Z^2 \setminus \{\bv\}$ agree with the lattice points in
        \[
        Q = T \cap H_{\ba}(\langle \ba, \bv\rangle+1)^{+}
        \]
        where $H_{\ba}(\alpha):= \{ \bx \in \R^d \colon \langle \ba,\bx \rangle = \alpha \}$ and $H_\ba(\alpha)^{+}$ denotes the positive halfspace. $Q$ is either a rational quadrilateral or a segment, the coordinates in its description have encoding length in $O(c)$, and $\bw_1$ can be found by minimizing $\langle \ba, \cdot \rangle$ over $Q \cap \Z^2$. Finding further local lattice diameter segments at $\bv$ that pass through points $\bw_2$ or $\bw_3$ can be done similarly, by intersecting $Q$ with the halfspace $H_{\ba}(\langle \ba, \bw_{i-1}\rangle+1)^{+}$, and with an additional check if $\bw_i \in T \cap H_{\ba}( \langle \ba, \bw_{i-1} \rangle)$. By \cite{ILP-2d-poly}, the number of arithmetic operations in integer linear programming over an $m$-gon in the plane is $O(m+c)$ where $c$ is the maximum binary encoding length of an input coordinate. In our case $m\leq 4$, so each iteration in this step has complexity $O(c)$. In total, this loop is repeated $|\Tcal|$ times, so it has complexity $O(nc)$.
        \item The number of lattice points on a segment $L \cap T$, where $L = \aff(\{\bv,\bw\})$, for $\bw,\bv \in \Z^2$ such that $\bw-\bv$ is primitive, is given by
        \begin{align}\label{eq:num_lp}
        |L \cap T \cap \mathbb{Z}^2| = \left\lfloor \frac{\vol_1(L \cap T)}{\|\bw - \bv\|} \right\rfloor + 1.
        \end{align}
        $|\mathcal{S}|$ is bounded above by $3\cdot |\Tcal| \leq 6n$. Finding the segment in $\Scal$ maximizing \eqref{eq:num_lp} can be done in $O(n)$ arithmetic operations.
    \end{enumerate}\vspace{-1.5em}

\end{proof}

\begin{remark}\label{rem:alg-in-P-P}
    We explain how to interpret the algorithm in the difference body $P-P$. Every triangle $T \in \Tcal$ considered in \Cref{alg:2d} can be embedded in the difference body $P-P$ as the triangle $T-\bv$, see \Cref{fig:triangles_in_P-P}. By \Cref{lemma:nvol} a local lattice diameter segment at $\bv$ has the largest lattice normalized length among all lattice segments through $\bv$. If it is a lattice diameter segment, then $\by-\bv = \diam_\Q(P)\bu$ for primitive $\bu$, so
    \[ \bu \in \frac{1}{\diam_\Q(P)}(P-P) = \lambda_1(P-P)\cdot (P-P).
    \]
    In other words, $\pm \bu$ are the first vectors obtained in a minimal scaling $\lambda$ of $(P-P)$ such that $\lambda(P-P) \cap \Z^2 \setminus \{\mathbf{0}\} \neq \emptyset$.
\end{remark}

\begin{figure}[ht!]
\centering
\includegraphics[width=\linewidth]{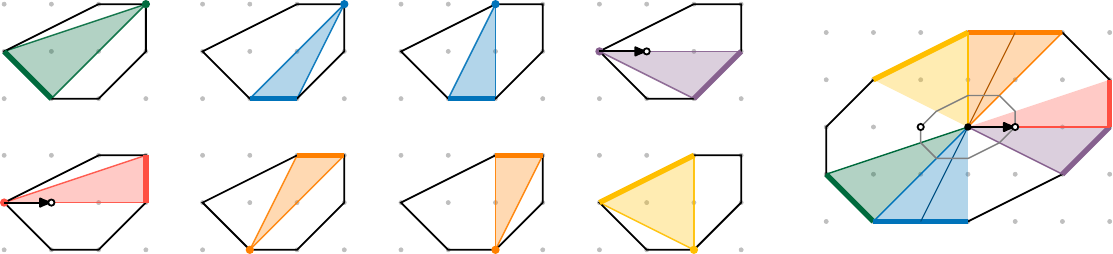}
\caption{Left: The eight triangles $\Tcal$ from \Cref{alg:2d} highlighted in eight copies of $P$. A primitive lattice diameter direction $\bu=(1,0)$ is marked with a black arrow. Right: The corresponding triangles are embedded in $P-P$, and the shortest vector $\bu = (1,0)$ lies on the boundary of $\frac{1}{3}(P-P)$. The lattice diameter of $P$ is $3$, and the unique lattice diameter direction is $(1,0)$.}
\label{fig:triangles_in_P-P}
\end{figure}

\begin{remark}
    A key ingredient of \Cref{alg:2d} is that lattice diameter segments pass through vertices. This fails in higher dimensions, as the following example shows. For $\bv_1 = (-1,0), \, \bv_2 = (0,-1), \, \bv_3 = (1,1)$, let
\[
P=\conv (\{ (-m,\bv_1), (-m,\bv_2), (-m+1,\bv_3),(m,-\bv_1),(m,-\bv_2),(m-1,-\bv_3)\}),
\]
where $m \geq 3$. The unique lattice diameter segment of $P$ lies on the first-coordinate axis, from values $-(m-1)$ to $m-1$. It does not pass through a vertex of $P$, not even through a lattice point on the boundary of $P$. In particular, the first coordinate axis contains a rational diameter segment yielding $\diam_\Q(P) = 2(m - \frac{1}{3})$ whereas $\ldiam(P) = 2m-2$ and so $\lfloor \diam_\Q(P) \rfloor \neq \ldiam(P)$. \Cref{fig:3d-novertex} depicts this polytope for $m=3$, where $\diam_\Q(P)=6-\frac{2}{3}$ and $\ldiam(P)=4$.
\begin{figure}[ht!]
    \centering
    \includegraphics{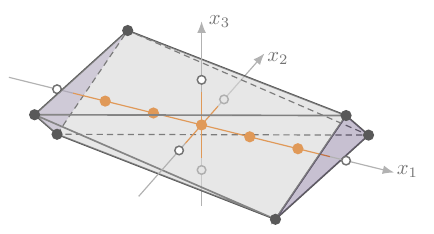}
    \caption{A \textit{unique} lattice diameter segment contained in the interior of a lattice polytope. The first lattice points on the axes not contained in $P$ are marked ``hollow''.}
    \label{fig:3d-novertex}
\end{figure}
\end{remark}

Our algorithms to compute the lattice diameter in general dimension and in dimension two are implemented and can be found in \cite{Code_LatticeDiameter}.

\section{Hardness of computing lattice diameters}\label{sec:NP-hard}

\subsection{Hardness in non-fixed dimension}\label{subsec:NP-hard-non-fixed-d}

When the dimension is non-fixed, we show that it is hard to compute the lattice diameter of centrally symmetric polytopes, even just of lattice cross-polytopes, i.e., polytopes of the form $\bB \cdot \Diamond_d$ where $\bB \in \Z^{d \times d}$ is some invertible matrix and $\Diamond_d := \conv(\pm \be_1, \ldots \pm \be_d)$ is the standard cross-polytope. For a $0$-symmetric convex body $K$, i.e., $K=-K$, we denote by $\| \bu\|_K$ the norm with the unit ball being $K$, i.e., $\|\bu \|_K := \min \{ \lambda > 0 \colon \bu \in \lambda K \} $. The $\ell_1$ norm, which we denote by $\| \bu\|_1$, is the norm whose unit ball is the standard cross polytope $\Diamond_d$. The \textit{shortest-vector problem} for a $0$-symmetric convex body $K$ and an invertible matrix $\bA \in \Z^{d \times d}$ is to determine a minimizer of
\begin{align*}
    \min_{\bu \in \bA \cdot \Z^d\setminus \{\mathbf{0}\}} \|\bu\|_K
\end{align*}
and it is denoted by $\SVP(K,\bA)$. We also denote the length of the shortest vector this way.
Regev and Rosen showed the following in \cite[Corollary~4.3]{Regev_Rosen}.

\begin{lemma}[Regev \& Rosen, 2006]\label[lemma]{lemma:Regev-Rosen}
    Let $\bA \in \Z^{d \times d}$ with $\det(\bA) \neq 0$. It is hard to approximate the shortest-vector problem
    \[	         \SVP(\Diamond_d,\bA) = \min_{\bu \in \bA \cdot \Z^d \setminus \{ \mathbf{0}\}} \|\bu\|_1.
	\]
    within any constant factor unless $\NPclass \subseteq \RPclass$.
\end{lemma}

A decision problem lies in $\RPclass$ if there exists a polynomial-time randomized algorithm that, if the correct answer is no, outputs \emph{no} with probability $1$, and if the correct answer is yes, outputs \emph{yes} with probability at least $1/2$. See \cite[Chapter 7]{AroraBarak} for more background. We would like to comment on the meaning of the hardness assertion in \Cref{lemma:Regev-Rosen}, which we have borrowed from the original source. In the study of SVP problems, one typically compares problems via randomized reductions rather than (standard) polynomial-time reductions. More technically, Regev \& Rosen show that approximation of the SVP for the $\ell_1$-norm within any given constant factor $c\ge 1$ is $\NP$-hard for randomized reductions. If $\Pclass = \RPclass$, which is widely assumed, then $\NP$-hardness under randomized reductions is equivalent to the (standard) $\NP$-hardness. 

\begin{lemma}\label{lemma:SVP-A-B-equiv}
    The problem of computing $\SVP(\Diamond_d, \bA)$ with an invertible matrix $\bA \in \Q^{d \times d}$, and the problem of computing $\SVP(\Diamond_d, \bB^{-1})$ with an integer invertible matrix $\bB \in \Z^{d \times d}$ are equivalent in terms of the polynomial-time reduction. 
\end{lemma}

\begin{proof}
    Clearly, if we can solve $\SVP(\Diamond_d, \bA)$ for any invertible matrix $\bA \in \Q^{d \times d}$, then we have also solved $\SVP(\Diamond_d, \bB^{-1})$ for any invertible matrix $\bB \in \Z^{d \times d}$. Conversely, for $\bA \in \Q^{d \times d}$, let $m$ be the least common multiple of the denominators of $\bA^{-1}$ (which can be computed in polynomial time), then for $\bB = m\cdot\bA^{-1}$
    \[
    \SVP(\Diamond_d, \bA) = \SVP(\Diamond_d, m\cdot \bB^{-1}) = m \cdot \SVP(\Diamond_d, \bB^{-1}). 
    \]
\end{proof}
	
    Next, we connect the shortest-vector problem with the problem of computing the lattice diameter. 
	\begin{lemma}\label[lemma]{lemma:diam:vs:svp}
		Let $\bB \in \Z^{d \times d}$ be an invertible matrix and let $P = \bB \cdot \Diamond_d$. Then 
		\[
			\frac{2}{\ldiam(P)+2} \le \SVP(\Diamond_d, \bB^{-1}) \le \frac{2}{\ldiam(P)}.
		\]
	\end{lemma} 

\begin{proof}
    Let $s = \ldiam(P)$. 
    Assume that the lattice diameter of $P$ is attained on the segment $[\bx,\by]$. Then $\by - \bx = s \bu$ for some primitive vector $\bu$. Since $P$ is $0$-symmetric, $\frac{1}{2} s  \bu = \frac{1}{2}  (\by-\bx) \in P = \bB  \cdot                                \Diamond_d$, which implies that $\bB^{-1}\bu$ is contained in $2/s \cdot \Diamond_d$, that is, $\SVP( \Diamond_d, \bB^{-1}) \leq 2/s$. This shows the upper bound on the optimum of the shortest-vector problem. On the other hand, if $\alpha := 1/ \SVP(\Diamond_d, \bB^{-1})$, then $\alpha \bB^{-1} \bu \in \Diamond_d $ for some $\bu \in \Z^d \setminus \{\mathbf{0}\}$.  This implies that $\alpha \bu \in P$. Consequently, $[ -k \bu, k \bu]$ with $k = \left\lfloor \alpha \right \rfloor$ is a lattice segment of lattice length $2 k$. One has $\ldiam(P) \ge 2k  \ge 2 \alpha - 2 = \frac{2}{\SVP(\Diamond_d, \bB^{-1})} -2$, giving the desired lower bound on $\SVP(\Diamond_d, \bB^{-1})$.
	\end{proof}

    We conclude with:
    
    \begin{theorem}\label{thm:NP-hard-crosspolytope}
        There exists a polynomial-time reduction from a factor-two approximation of $\SVP$ with respect to the $\ell_1$-norm for a given lattice to computation of the lattice diameter of a given $0$-symmetric lattice cross-polytope. In particular, it is hard to approximate the lattice diameter of a $0$-symmetric lattice cross-polytope within a constant factor unless $\NPclass \subseteq \RPclass$.
	\end{theorem}

    \begin{proof}
        Consider the problem $\SVP(\Diamond_d, \bB^{-1})$ for a given $\bB \in \Z^{d \times d}$. The reduction maps $\bB$ to the polytope $P = \bB \cdot \Diamond_d$, the input for the lattice-diameter computation.  For a $0$-symmetric lattice cross-polytope $P$, by computing $\ldiam(P)$  one can also determine the value $2 / (\ldiam(P) +2)$, which (by \Cref{lemma:diam:vs:svp}) approximates $\SVP(\Diamond_d, \bB^{-1})$ within the factor-two, because $2 / \ldiam(P)$ is at most twice as large as $2 / (\ldiam(P) +2)$ in view of $\ldiam(P) \ge 2$. The hardness conclusion follows directly from \Cref{lemma:Regev-Rosen}.
        \end{proof}

	\begin{remark}The same argument can also be made for hyper-parallelepipeds, relying on the hardness of approximation  result for the shortest-vector problem with respect to the $\ell_\infty$ norm, see \cite{Boas_SVP_linf}.
	\end{remark} 

\subsection{Hardness for semi-algebraic sets in fixed dimension}

We show that even in fixed dimension, computing lattice diameters is $\NP$-hard for semi-algebraic sets that are known to contain lattice points. We reduce to the following $\NP$-complete problem.

\begin{lemma}[Adleman and Manders]\label[lemma]{lemma:NP-complete-abc}
    Given $a,b,c \in \Z_{\geq 1}$, it is $\NPclass$-complete to decide if $x^2 \equiv a \pmod b$ holds for some $x \in \{1, \ldots, c-1\}$.
\end{lemma}

This problem can be found in \cite[AN1]{GJ} which provides an account of important $\NP$-hard problems that occur across different areas of mathematics. It stays $\NP$-complete if we allow $x \in \{1, \ldots, c\}$ and restrict to $b,c \in \Z_{\geq k}$ for any constant $k \in \Z_{\geq 1}$, since for bounded $b$ and $c$, the decision problem can be solved in polynomial time.
From this, we deduce the next theorem.

\begin{theorem}\label{thm:NP-hard-semialgebraic-set}
    It is $\NP$-hard to compute the lattice diameter of a two-dimensional compact basic closed semi-algebraic set $K \subset \R^2$ that satisfies $K \cap \Z^2 \neq \emptyset$. It stays $\NP$-hard even to compute the lattice diameter in the fixed direction $(0,1)$.
\end{theorem}

\begin{proof}
    Let $b,c \in \Z_{\geq 4}$. The congruence $x^2 \equiv a \pmod b$ from \Cref{lemma:NP-complete-abc} is satisfied if and only if there exists $x \in \{1, \ldots, c\}$ and $y \in \Z$ such that $f(x,y): = x^2-a-by =0$. We construct a semi-algebraic set $K$ with lattice diameter in direction $(0,1)$, such that $\ldiam(K)>1$ if and only if $f(x,y) =0$ for some $x \in \{1, \ldots, c\}$ and $y \in \Z$. Let $K$ be given as
    \begin{align*}
        K:= \{ (x,y) \in \R^2 \colon 1 \leq x \leq c, \ f(x,y)^2f(x,y+c)^2g(x,y) \geq 0 \}
    \end{align*}
    with $g(x,y):= (x-1)(2-x) + (y-1)(2-y)$. The inequality $g(x,y) \geq 0$ describes a Euclidean disc $B$ containing only the lattice points $\{1,2\}^2$.
    The inequalities $1 \leq x \leq c$ describe the strip $S=[1,c] \times \R$, and the equation $f(x,y)=0$ and $f(x,y+c)=0$ describe two parabolas $C$ and $C' = C-(0,c)$. We have $K = S \cap (C \cup C' \cup B)$, i.e., $K$ is the union of the parabolas $C, C'$ and the disc $B$, cropped by the strip $S$. Since $b,c \geq 4$, none of the points $\{1,2\}^2$ is on the parabolas $C$ or $C'$, and clearly $C \cap S$ contains a lattice point if and only if $C' \cap S$ does because they are translations of each other by the lattice vector $(0,c)$. If $C \cap S$ and $C' \cap S$ do not contain lattice points, then $\ldiam(K)=\ldiam(B) =1$. Otherwise, if $C \cap S$ contains a lattice point $(x,y)$, then $(x,y-c) \in C' \cap S$, and $[(x,y), (x,y-c)]$ is a lattice diameter segment of $K$ in direction $(0,1)$, yielding $\ldiam(K) = c$. Thus, we have a polynomial-time reduction of AN1 to computation of the lattice diameter of $K$ in the fixed direction $(0,1)$.
\end{proof}

\begin{remark}
    We leave it as an open problem to determine if the decision version of this problem lies in $\NPclass$. The answer might be negative, because the similar problem of deciding whether a (not necessarily bounded) semi-algebraic set contains a lattice point is undecidable, see \cite{Lemma2.1,jeroslow1973,matiyasevich1993_hilbert10}.
\end{remark}

\section{An Ehrhart-type function counting lattice diameter segments}\label{sec:QP}

Ehrhart theory studies the number of lattice points in integer dilations of a given lattice (or rational) polytope. A fundamental result by Ehrhart states that this counting function agrees with a polynomial function for lattice polytopes, and a quasi-polynomial function for rational polytopes. For details, see \cite{beckrobins}. While counting the number of all lattice segments in a fixed direction that intersect a polytope $P$ in a lattice point can be easily done, counting just those that are lattice diameter segments is much more difficult. Surprisingly, we have an Ehrhart-type result.

We use some results from \textit{Presburger arithmetic}, also known as \textit{linear integer arithmetic}. 

\begin{definition}
    A \emph{Presburger formula} is a first-order formula in the language of addition, evaluated over the integers. A \emph{Presburger counting function} $g_F:\N^n \to \N$ for a Presburger formula $F(\bp,\bs)$ is $g_F(\bs) := \# \{ \bp \in \N^d \colon F(\bp,\bs) \ \text{holds}\}$.
\end{definition}

\begin{remark}
First we explain some preliminaries about \textit{Presburger arithmetic}. A first-order formula in addition over the integers can be described as follows. A \textit{term} is an expression of the form $a_1x_1 + \ldots + 
a_nx_n$ for $a_1, \ldots, a_n \in \Z$, $n \in \N$ where $x_1, \ldots, x_n$ are free variables. An \textit{atom} is a \textit{formula} of the form $t_1=t_2, t_1 < t_2$ or $t_1 \equiv t_2 \pmod n$, for terms $t_1,t_2$ and an $n \in \N$. Any \textit{formula} can inductively be defined as $\varphi \wedge \psi$, $\varphi \vee \psi$, $\neg \varphi$, $\exists x \ \varphi$, or $\forall x \ \varphi $ where $\varphi,\psi$ are formulas (which may or may not contain the quantified variable). Notice that $t_1 \leq t_2$ is equivalent to $(t_1 < t_2) \vee (t_1 = t_2)$. A vector $\bu = (u_1, \ldots, u_n) \in \Z^n$ can either satisfy, or not satisfy, a formula in the natural way by evaluating terms in the free variables $x_1, \ldots, x_n$ on the coordinates of $\bu$. For an introduction to Presburger arithmetic, see \cite{Chistikov_FOL}.

\end{remark}

\begin{remark}\label[remark]{rem:Presburger-N-to-Z}
For simplicity, Presburger arithmetic is sometimes defined over the natural numbers $\N$ only with addition, instead of over $\Z$. The formulas attainable this way are the same set of formulas attainable as in our definition above. To see this, if a coefficient $a_i$ used in terms is negative, then rearranging summands between the terms $t_1,t_2$ in $t_1 = t_2, t_1 < t_2$ or $t_1 \equiv t_2$, yields an equivalent formula with only positive coefficients. For example, $2x +(-3)y =z$ is equivalent to $2x = 3y+z$. Similarly, if $u \in \Z$ satisfies a formula $\varphi$, then one can rearrange terms in $\varphi$ to get an equivalent formula that is satisfied by $-u$, let $\varphi^-$ denote this formula. Then $\exists x \in \Z \colon \varphi$ if and only if $\exists x \in \N \colon \varphi \vee \varphi^-$.
\end{remark}

Kevin Woods proved the following lemma in \cite[Theorem~1.10]{Woods_Presburger}.

\begin{lemma}[Woods, 2015] \label[lemma]{lem:KevinWoods}
 A Presburger counting function $g:\N^n \to \Q$ is a piecewise quasi-polynomial.
\end{lemma}

For $n=1$, a function $g:\N \to \Q$ is a piecewise quasi-polynomial if and only if it \emph{eventually} agrees with a quasi-polynomial, i.e., if there exists an integer $N \in \N$ such that $g: \N_{k \geq N} \to \Q$ agrees with a quasi-polynomial. Thus it suffices to express the number of lattice diameter segments in a fixed direction as a Presburger counting function. The crucial step is to show that the set of all lattice diameter directions of dilates of a lattice polytope is finite.

\begin{lemma}\label[lemma]{lemma:ld-direction-finite}
    Let $P \subset \R^d$ be a lattice $d$-polytope. Then there exists a finite set $D$ of primitive vectors that depends only on $P$ with the property that for each $k \in \Z_{\ge 1}$ and each $\bx,\by \in \Z^d \cap k P$ satisfying $\by - \bx = \ldiam(kP) \bu$ for some primitive vector $\bu$, one necessarily has $\bu \in D$. 
\end{lemma} 
\begin{proof} 
Pikhurko showed in \cite{Pikhurko2001} that there exists a constant $\varepsilon_d>0$, depending only on the dimension, with the property that whenever a lattice polytope $P$ in $\R^d$ has lattice points in the interior, then some of these lattice points lie deeply inside $P$, with the depth quantified in terms of $\varepsilon_d$. Formally, this can be phrased as the existence of $p \in \Z^d \cap P$ with $p + \varepsilon_d (P-P) \subseteq P$.


First, suppose that $k \geq d+1$. Then $k P$ has interior lattice points. Thus, we can choose $\bx(k) \in \Z^d \cap k P$ with $\bx(k) + \varepsilon_d k (P-P) \subseteq kP $. Consider the first successive minimum $\lambda = \lambda_1(P-P)$ of $P-P$, defined as the smallest $\lambda$ such that $\Z^d \cap \lambda \cdot (P-P) \ne \{0\}$. The set $\lambda \cdot (P-P)$ has exactly one interior lattice point with the boundary of $\lambda \cdot (P-P)$ containing a primitive vector, which we denote by $\bu$. We fix $\by(k) = \bx(k) + \lfloor \frac{\varepsilon_d k}{\lambda} \rfloor \bu \in k P$. One has $\ldiam(k P) \ge |[\bx(k),\by(k)] \cap \Z^d | -1 \ge \lfloor \frac{\varepsilon_d k}{\lambda} \rfloor \ge \frac{\varepsilon_d k}{\lambda} -1.$ Consequently, $\lambda \ge \frac{\varepsilon_d k}{1 + \ldiam(kP) } \geq \frac{\varepsilon_d k }{ 2 \ldiam(kP)}$. 

Now, if $\bx, \by \in \Z^d \cap k P$ are such that $\by - \bx = \ldiam(kP) \cdot \bu$ for some primitive vector, then $\bu = \frac{1}{\ldiam(kP)} (\by-\bx) \in \frac{k}{\ldiam(k P) } (P-P)$. Consequently, $\bu \in \frac{2 \lambda}{\varepsilon_d} \cdot (P-P)$, that is, $\bu$ is in a set that does not depend on $k$, but only on $P$. So, the assertion holds for $D$ containing all lattice diameter directions of the polytopes $P, 2 P,\ldots, d P$ and all primitive vectors of $\frac{2 \lambda_1(P-P)}{\varepsilon_d} \cdot (P-P)$. 
\end{proof} 

\begin{theorem}\label{thm:QP-all-dimensions}
    Let $P \subset \R^d$ be a $d$-dimensional lattice polytope. Then the counting function $\LD_P(k):\N \to \N$ eventually agrees with a quasi-polynomial in the dilation variable.
\end{theorem}

\begin{proof}
     Let $P$ be described by a system of linear inequalities $\bA \bx \le \bb$. Then $k P$ with $k \in \Z_{\ge 1}$ is described by the system $\bA \bx \le k \bb$. Let $D$ be the set of lattice diameter directions of $kP$ for any $k \in \N$, as in \Cref{lemma:ld-direction-finite}, and fix $\bu \in D$. For $s \in \N$ the condition that $s \le \ldiam(kP)$ can be expressed by the Presburger formula
\[
	F_\bu(s,k):= \ ( \exists \bx,
    \by \in \Z^d \colon  (\bA \bx \le k \bb) \wedge (\bA\by \le k \bb) \wedge (\by - \bx = s \bu) \bigr). 
\]
Then, the equality condition $s = \ldiam(kP)$ can be expressed by 
\[
	G_\bu(s,k) := \  F_\bu(s,k) \wedge \bigwedge_{\bu' \in D}\overline{F_{\bu'}(s+1,k)}, 
\] where the bar denotes the negation.
Fix a lattice projection along $\bu$ as $\pi_\bu: \R^d \to \R^{d-1}$ with one-dimensional kernel spanned by $\bu$ and $\pi_\bu(\Z^d) = \Z^{d-1}$. Now we can express the fact that, for $\bp \in \Z^{d-1}$, the line $L_\bp := \{ \bx \in \R^d \colon \pi_\bu(\bx) = \bp\}$ is a lattice diameter line as 
\begin{equation} \label{eq:H_def}
\begin{aligned}
H_\bu(\bp,k)& := \exists s \in \N \;:\;
\bigl( G_\bu(s,k) \wedge \exists \bx,\by \in \Z^d : \\[2pt] & (\bA\bx \le k\bb) \wedge (\bA\by \le k\bb)  \wedge (\by-\bx = s\bu) \wedge (\pi_\bu(\bx) = \pi_\bu(\by) = \bp) \bigr).
\end{aligned}
\end{equation}

The number of lattice diameter segments of $kP$ in direction $\bu$ is given by the Presburger counting function
\begin{align*}
   \LD_{P,\bu}(k) = \# \{ \bp \in \Z^{d-1} \colon H_\bu(\bp,k) \  \text{ holds}\}
\end{align*}
and by \Cref{lem:KevinWoods} $\LD_{P,\bu}(k)$ eventually agrees with a quasi-polynomial in $k$.
\begin{align*}
    \LD_P(k) = \sum_{\bu \in D} \LD_{P,\bu}(k)
\end{align*}
also eventually agrees with a quasi-polynomial.
\end{proof}

For lattice polygons $P$, one can show that $\LD_P(k)$ is quasi-linear, and its period can be expressed in terms of edge slopes of the lattice polygon. We demonstrate this in the following example. Notice that the leading coefficient depends on the residue class; it is not an invariant as in the case of the Ehrhart polynomial for rational polytopes.
    
 \begin{example}\label{ex:QP}
 Let $P= \conv(\{(0,0), (1,3), (5,1), (6,4)\})$. \Cref{fig:dilations} depicts $P$ and its dilates with their respective lattice diameter segments.
 \begin{figure}[ht!]
    \centering    \includegraphics[width=\linewidth]{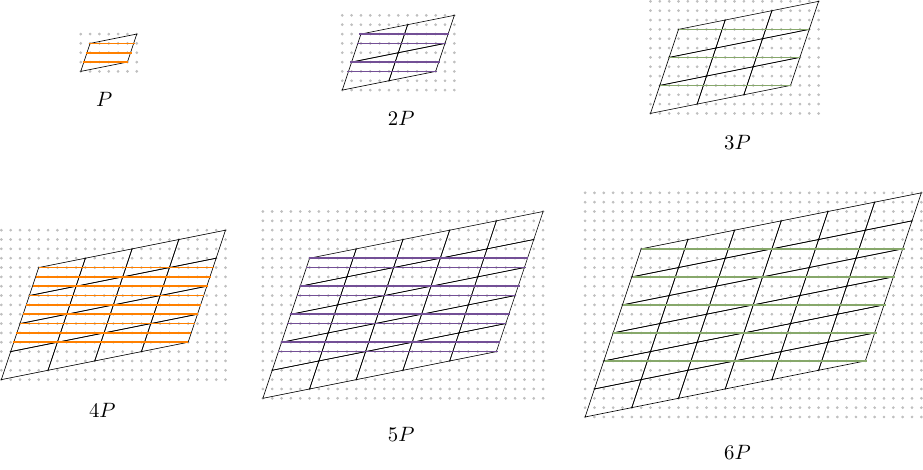}
    \caption{Lattice diameter segments in dilations of $P$, repeating periodically depending on the dilation factor.}
    \label{fig:dilations}
\end{figure}
 
For all $k \in \N$ the only lattice diameter direction is $\bu = (1,0)$. Furthermore, all lattice diameter segments of $kP$ are contained in the rational parallelogram $kR = k\cdot \conv (\{ (1/3,1), (1,3), (5,1), (17/3,3)\})$. The slope of the left and right edges of $R$ is $q=3$. 
    Let $i \in \{0, 1,2\}$ and $k \equiv i \pmod 3$. We can partition the set of parallel lattice segments in direction $\bu$ into blocks of $q$ consecutive lines, which we call $q$-blocks. For fixed $i$, the number of full $q$-blocks in $kR$, denoted $\blocks_i(k)$ is a linear function in $k$. Which lines of a $q$-block are lattice diameter lines is independent of the block, and independent of $k$, so let $n_i$ denote this number. Similarly, the number of remaining lattice diameter segments in $kR$ in direction $\bu$ that are not contained in a full $q$-block is independent of $k$; denote it with $r_i$. The counting function can be computed as 
    \begin{align*}
        \LD_P(k) = n_i \cdot \blocks_i(k) + r_i.
    \end{align*}
    In our example, we have the following values:
\begin{center}
    \begin{tabular}{c|c|c|c}
    $i$ & $n_i$ & $\blocks_i(k)$ & $r_i$\\ \hline \rule{0pt}{2.5ex}
    1 \hspace{1pt} & 3 & $\tfrac{1}{3}(2k+1)$ & 0  \\[3pt]
    2 & 2 & $\tfrac{1}{3}(2k-1)$  & 2 \\[3pt]
    0 & 1 & $\tfrac{1}{3}(2k)$  & \ 1,
    \end{tabular}\vspace{-1em}
\end{center}

which yields
\begin{align*}
        \LD_P(k) = \begin{cases}
         2k+1 & \text{if } k \equiv 1 \pmod 3,\\
            \frac{4}{3} k + \frac{4}{3} & \text{if } k \equiv 2 \pmod 3,\\
            \frac{2}{3} k + 1 & \text{if } k \equiv 0 \pmod 3.
            \end{cases}
    \end{align*}
 \end{example}

 We conjecture that the counting function of $\ell$-dimensional slices of $P$ that contain the most lattice points among all $\ell$-dimensional slices of $P$ has a similar behavior under dilation of $P$. More precisely:
 
\begin{conjecture}
   Let $P$ be a $d$-dimensional lattice polytope, and $2 \leq \ell \leq d-1$. Define $\LD_P^{\ell}(k)$ as the number of $\ell$-dimensional sections of $kP$ that contain the most lattice points among all $\ell$-dimensional sections of $kP$. Then $\LD_P^{\ell}:\N \to \N$ eventually agrees with a quasi-polynomial function of degree $d-\ell$.
\end{conjecture}

\section{The discrete Borsuk partition problem}\label{sec:Borsuk}

In this section, we prove the discrete version of Borsuk's partition problem for lattice diameters. More generally, we study lattice diameters of bounded sets $S \subset \Z^d$. 
Notice that the function $f : \Z^d \times \Z^d \to \R$ defined by  
\[
f(\bx, \by) = |[\bx, \by] \cap \Z^d| - 1
\]  
    is a semi-metric, i.e., it is positive-definite and symmetric, so $(\Z^d, f)$ is a semi-metric space. Furthermore, $f$ satisfies $f(\lambda \bx, \lambda \by) = |\lambda| \cdot f(\bx,\by)$ for all $\bx,\by \in \Z^d$ and all $\lambda \in \Z$. From this viewpoint, \Cref{thm:Borsuk} is closely related to the well-known Boltyansky-Gohberg conjecture \cite{boltjansky1985results} which states that any bounded set $S$ in a $d$-dimensional normed space can be partitioned into at most $2^d$ subsets, each having strictly smaller diameter.

Recall that for a finite $S \subset \Z^d$ the \emph{lattice Borsuk number} of $S$, denoted $\beta_\Z(S)$, is the smallest number of subsets into which $S$ can be partitioned such that each subset has strictly smaller lattice diameter than $\diam_{\Z} (S)$. In \mbox{\Cref{fig:Borsuk-sets}} we show a Borsuk partition of a set $S=P \cap \Z^d$, where $P$ is a lattice polygon, into parts $S_1$ and $S_2$. Necessarily $\ldiam(S_i)< \ldiam(P)$ for $i=1,2$, but 
note that $\ldiam (\conv (S_1)) = \ldiam(P)$.

\begin{figure}[ht!]
    \centering
     \includegraphics{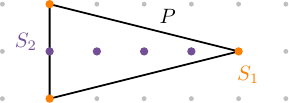}
    \caption{A Borsuk partition $P \cap \Z^2 = S_1 \sqcup S_2$ where $\ldiam(S_1) < \ldiam(\conv(S_1))=\ldiam(P)$.}
    \label{fig:Borsuk-sets}
\end{figure}

Define the \emph{lattice Borsuk graph} of $S$, denoted $\BG_\Z(S)$, as the graph with vertex set $S$ and edges $\{\bx,\by\} \subseteq S$ if and only if $f(\bx,\by)=\diam_\Z(S)$.

\begin{remark}\label{rem:equivalence}
    For a finite set $S \subset \Z^d$,  $\beta_\Z(S)=\chi(\BG_\Z(S))$, where $\chi$ is the \textit{chromatic number} of $\BG_\Z(S)$, i.e., the minimum number of colors needed to color $S$ so that no two adjacent vertices in $S$ have the same color.
\end{remark}

\begin{lemma}\label[lemma]{lemma:CompleteGraphs}
    Let $S\subset\Z^d$ be a bounded set. Then every connected component of $\BG_\Z(S)$ is a clique with at most $2^d$ vertices.
\end{lemma}

\begin{remark}\label[remark]{rem:at-most-three-lld}
    In particular, \Cref{lemma:CompleteGraphs} says that the number of lattice diameter segments ending at any given point in $S$ is at most $2^d-1$. This generalizes what Alarcon showed in \cite[Lemma~2.4]{Alarcon_PhD} for dimension two to any dimension.
\end{remark}

\begin{proof}[Proof of \Cref{lemma:CompleteGraphs}]
    Let $m=\ldiam(S)$ and let $X$ be the vertex set of a connected component of $\BG_\Z(S)$. If $S$ is empty or $|S|=1$, then the claim is trivially true, so we can assume $m \geq 1$. Recall that edges in the Borsuk graph correspond to lattice diameter segments. Two vertices $\bx,\by \in X$ are endpoints of a lattice diameter segment if and only if $\by-\bx =m\bu$ for some $\bu \in \Z^d \setminus \{\mathbf{0}\}$, i.e., $\by-\bx \in m\Z^d$. Since $\BG_\Z(X)$ is connected, any two vertices $\bx,\by$ are connected by a path $\bx=\bz_0, \bz_1, \ldots, \bz_\ell = \by$. Each $\{\bz_i,\bz_{i+1}\}$ is an edge in the graph, so $\bz_i - \bz_{i+1} \in m\Z^d$ for all $i$, but then also $\by-\bx \in m\Z^d$. So $\{\bx,\by\}$ itself is an edge in $\BG_\Z(X)$.
    This proves that all connected components are cliques.
    For the size, consider $Y := (X-\bx)/m$ for some $\bx \in X$. This is a set of integer points containing the origin. Suppose $|Y| > 2^d$, then there exist distinct points $\by,\bz \in X$ with $(\by-\bx)/m - (\bz-\bx)/m \in 2 \Z^d$. But then $\by-\bz \in 2m\Z^d$, a contradiction to the choice of $m$. Hence $|X| \leq 2^d$.
\end{proof}

Finally, we present a solution to the discrete version of Borsuk's partition problem.

\begin{theorem}\label{thm:Borsuk}
     Let $S \subset \Z^d$ be a bounded set. Then $\beta_\Z(S) \leq 2^d$ and this bound is best possible.
\end{theorem}

\begin{proof}
    By \Cref{lemma:CompleteGraphs} each connected component of $\BG_\Z(S)$ is a complete graph with at most $2^d$ vertices. Therefore, 
    \[
    \beta_\Z(S) = \chi(\BG_\Z(S)) \leq  2^d.
    \]
    The set $S=\{0,1\}^d$ is an example where this bound is attained.
\end{proof}

\noindent \textbf{Acknowledgements.}
We are grateful to Imre Bárány,  Matthias Beck, Marie-Charlotte Brandenburg, Fritz Eisenbrand, Ansgar Freyer, Chiara Meroni, János Pach, and Matthias Schymura for fruitful discussions and suggestions, and to Bruce Reznick for providing important references. An extended abstract of this paper has appeared in \cite{IPCO_ourpaper}. The first author
was supported by the SPP 2458 ``Combinatorial Synergies'', funded by the Deutsche Forschungsgemeinschaft project 539867386.
The other authors were partially supported by NSF grants 2348578 and 2434665.

\noindent\textbf{The role of AI.} With the exception of spell-checking and assistance in making figures, AI was not used in the research and writing of this paper. The authors take full responsibility for the content, relevance and correctness of the results. 


\bibliographystyle{amsalpha} { \linespread{1.2} \small \bibliography{references} }

@misc{HollowP,
  author = {Arun, S. and Dillon, T.},
  title = {{H}ollow polytopes with many vertices},
  note = {arXiv:2504.17530},
  year = {2025}
}

@book{GLS1993,
  author    = {Gr{\"o}tschel, M. and Lov{\'a}sz, L. and Schrijver, A.},
  title     = {Geometric Algorithms and Combinatorial Optimization},
  series    = {Algorithms and Combinatorics},
  volume    = {2},
  edition   = {2nd},
  publisher = {Springer-Verlag},
  address   = {Berlin},
  year      = {1993},
  doi       = {10.1007/978-3-642-78240-4},
  isbn      = {978-3-642-78242-8}
}

@misc{AverkovetAl2026ExactFlatness,
  author        = {Averkov, G. and Codenotti, G. and
                   Freyer, A. and Huang, K.},
  title         = {Exact Flatness Constant for One-Point Convex Bodies
                   and the Discrete Isominwidth Problem: The Planar Case},
  year          = {2026},
  eprint        = {2604.27260},
  archivePrefix = {arXiv},
  primaryClass  = {math.MG},
  doi           = {10.48550/arXiv.2604.27260},
  url           = {https://arxiv.org/abs/2604.27260},
  note          = {arXiv:2604.27260}
}

@article{bubeck2015convex,
  author  = {Bubeck, S{\'e}bastien},
  title   = {Convex Optimization: Algorithms and Complexity},
  journal = {Foundations and Trends in Machine Learning},
  volume  = {8},
  number  = {3--4},
  pages   = {231--357},
  year    = {2015},
  doi     = {10.1561/2200000050}
}

@misc{Grinsztajn2026,
  author       = {M. Grinsztajn},
  title        = {A 63-dimensional counterexample to {Borsuk}'s conjecture},
  year         = {2026},
  howpublished = {\url{https://github.com/maaxgrin/borsuk-63-counterexample}},
  note         = {Proof note and exact verification script},
}

@article{levin1965algorithm,
  author  = {Levin, A. Yu.},
  title   = {On an algorithm for the minimization of convex functions},
  journal = {Soviet Mathematics Doklady},
  volume  = {6},
  pages   = {286--290},
  year    = {1965}
}

@phdthesis{Alarcon_PhD,
  author = {Alarcon II, E. G.},
  title = {{C}onvex {L}attice {P}olygons},
  year = {1995},
  school = {University of Illinois at Urbana-Champaign}
}

@article{arno80,
  author    = {V. I. Arnold},
  title     = {Statistics of integral convex polytopes},
  journal   = {Funkts. Anal. Pril.},
  volume    = {14},
  year      = {1980},
  pages     = {1--3},
  note      = {(in Russian), English translation: {Funct. Anal. Appl.} {14} (1980), 79--84}
}

@article{Barany_Furedi,
  author = {B{\'a}r{\'a}ny, I. and F{\"u}redi, Z.},
  title = {{O}n the lattice diameter of a convex polygon},
  journal = {Discrete Math.},
  volume = {241},
  number = {1--3},
  pages = {41--50},
  year = {2001},
  publisher = {Elsevier}
}

@book{boltjansky1985results,
  author = {Boltyansky, V. G. and Gohberg, I.},
  title = {{R}esults and problems in combinatorial geometry},
  year = {1985},
  publisher = {CUP Archive}
}

@article{borsuk1933drei,
  author = {Borsuk, K.},
  title = {{D}rei {S}{\"a}tze {\"u}ber die $n$-dimensionale {E}uklidische {S}ph{\"a}re},
  journal = {Fund. Math.},
  volume = {20},
  pages = {177--190},
  year = {1933},
  publisher = {Polska Akademia Nauk. Instytut Matematyczny PAN}
}

@article{bara92,
  author = {B{\'a}r{\'a}ny, I. and Pach, J.},
  title = {{O}n the number of convex lattice polygons},
  journal = {Comb. Probab. Comput.},
  volume = {1},
  year = {1992},
  pages = {295--302}
}

@book{beckrobins,
  author = {Beck, M. and Robins, S.},
  title = {{C}omputing the {C}ontinuous {D}iscretely: {I}nteger-point {E}numeration in {P}olyhedra},
  edition = {2nd},
  year = {2015},
  publisher = {Springer},
  series = {Undergraduate Texts in Mathematics}
}

@article{bara92a,
  author = {B{\'a}r{\'a}ny, I. and Vershik, A. M.},
  title = {{O}n the number of convex lattice polytopes},
  journal = {Geom. Funct. Anal.},
  volume = {2},
  year = {1992},
  pages = {381--393}
}

@phdthesis{Corzatt_PhD,
  author = {Corzatt, C. E.},
  title = {{S}ome {E}xtremal {P}roblems of {N}umber {T}heory and {G}eometry},
  year = {1974},
  school = {University of Illinois at Urbana-Champaign}
}

@article{Lemma2.1,
  author = {De Loera, J. A. and Hemmecke, R. and K{\"o}ppe, M. and Weismantel, R.},
  title = {{I}nteger {P}olynomial {O}ptimization in {F}ixed {D}imension},
  journal = {Math. Oper. Res.},
  volume = {31},
  number = {1},
  pages = {147--153},
  year = {2006},
  publisher = {INFORMS},
  doi = {10.1287/moor.1050.0169}
}

@book{matiyasevich1993_hilbert10,
  author    = {Y. Matiyasevich},
  title     = {Hilbert's Tenth Problem},
  publisher = {The MIT Press},
  address   = {Cambridge, MA and London},
  year      = {1993}
}

@article{jeroslow1973,
  author  = {R. C. Jeroslow},
  title   = {There Cannot be any Algorithm for Integer Programming with Quadratic Constraints},
  journal = {Oper. Res.},
  volume  = {21},
  number  = {1},
  pages   = {221--224},
  year    = {1973},
  doi     = {10.1287/opre.21.1.221}
}

@InProceedings{ILP-2d-poly,
  author = {Eisenbrand, F. and Laue, S.},
  editor = {Ibaraki, T. and Katoh, N. and Ono, H.},
  title = {{A} {F}aster {A}lgorithm for {T}wo-{V}ariable {I}nteger {P}rogramming},
  booktitle = {Algorithms and {C}omputation},
  series = {Lect. Notes Comput. Sci.},
    volume = {2906},
    year = {2003},
  publisher = {Springer},
  pages = {290--299}
}

@article{freyerhenk22,
  author = {Freyer, A. and Henk, M.},
  title = {{B}ounds on the lattice point enumerator via slices and projections},
  journal = {Discrete Comput. Geom.},
  volume = {67},
  year = {2022},
  pages = {895--918}
}

@article{freyerhenk24,
  author = {Freyer, A. and Henk, M.},
  title = {{P}olynomial bounds in {K}oldobsky's discrete slicing problem},
  journal = {Proc. Amer. Math. Soc.},
  volume = {152},
  number = {7},
  pages = {3063--3074},
  year = {2024}
}

@article{gardnergritzmann97,
  author = {Gardner, R. J. and Gritzmann, P.},
  title = {{D}iscrete tomography: {D}etermination of finite sets by {X}-rays},
  journal = {Trans. Amer. Math. Soc.},
  volume = {349},
  number = {6},
  year = {1997},
  pages = {2271--2295}
}

@article{gard05,
  author = {Gardner, R. J. and Gronchi, P. and Zong, C.},
  title = {{S}ums, projections, and sections of lattice sets, and the discrete covariogram},
  journal = {Discrete Comput. Geom.},
  volume = {34},
  year = {2005},
  pages = {391--409}
}

@book{GJ,
  author = {Garey, M. R. and Johnson, D. S.},
  title = {{C}omputers and {I}ntractability: {A} {G}uide to the {T}heory of {NP}-{C}ompleteness},
  publisher = {Freeman},
  year = {1979}
}

@book{gruber-lekkerkerker1987geometry,
  author = {Gruber, P. M. and Lekkerkerker, C. G.},
  title = {{G}eometry of {N}umbers},
  publisher = {North-Holland},
  year = {1987},
  edition = {2nd},
  volume = {37}
}

@article{grunbaum1957simple,
  author  = {B. Gr{\"u}nbaum},
  title   = {{A} simple proof of {B}orsuk's conjecture in three dimensions},
  journal = {Math. Proc. Cambridge Philos. Soc.},
  volume  = {53},
  pages   = {776--778},
  year    = {1957}
}

@article{kahn1993counterexample,
  author = {Kahn, J. and Kalai, G.},
  title = {{A} counterexample to {B}orsuk’s conjecture},
  journal = {Bull. Amer. Math. Soc.},
  volume = {29},
  pages = {60--62},
  year = {1993}
}

@article{ILP_poly,
  author = {Lenstra, H. W.},
  title = {{I}nteger {P}rogramming with a {F}ixed {N}umber of {V}ariables},
  journal = {Math. Oper. Res.},
  volume = {8},
  number = {4},
  pages = {538--548},
  year = {1983}
}

@article{latte1,
  author = {De Loera, J. A. and Hemmecke, R. and Tauzer, J. and Yoshida, R.},
  title = {{E}ffective lattice point counting in rational convex polytopes},
  journal = {J. Symbolic Comput.},
  volume = {38},
  number = {4},
  year = {2004},
  pages = {1273--1302}
}

@article{lagariasZiegler1991bounds,
  author = {Lagarias, J. C. and Ziegler, G. M.},
  title = {{B}ounds for lattice polytopes containing a fixed number of interior points in a sublattice},
  journal = {Canad. J. Math.},
  volume = {43},
  number = {5},
  year = {1991},
  pages = {1022--1035}
}

@article{rabi89,
  author = {Rabinowitz, S.},
  title = {{A} {T}heorem about collinear lattice points},
  journal = {Util. Math.},
  volume = {36},
  year = {1989},
  pages = {93--95}
}

@book{Schrijver,
  author = {Schrijver, A.},
  title = {{T}heory of {L}inear and {I}nteger {P}rogramming},
  year = {1986},
  publisher = {Wiley-Interscience}
}

@article{barvinok1994polytime,
  author  = {Barvinok, A.},
  title   = {A polynomial time algorithm for counting integral points in polyhedra when the dimension is fixed},
  journal = {Math. Oper. Res.},
  volume  = {19},
  year    = {1994},
  pages   = {769--779}
}

@inproceedings{Regev_Rosen,
  author = {Regev, O. and Rosen, R.},
  title = {Lattice problems and norm embeddings},
  booktitle = {Proc. 38th ACM Symp. Theory Comput. (STOC)},
  pages = {447--456},
  year = {2006}
}

@book{Boas_SVP_linf,
  title={Another {NP}-complete partition problem and the complexity of computing short vectors in a lattice},
  author={van Emde Boas, P.},
  series={Report. Department of Mathematics. University of Amsterdam},
  url={https://books.google.com/books?id=tCQiHQAACAAJ},
  year={1981},
}

@article{Woods_Presburger,
  author = {K. Woods},
  title = {Presburger arithmetic, rational generating functions, and quasi-polynomials},
  journal = {J. Symb. Log.},
  volume = {80},
  number = {2},
  pages = {433--449},
  year = {2015},
  publisher = {Assoc. for Symbolic Logic}
}

@article{Pikhurko2001,
  author  = {Pikhurko, O.},
  title   = {Lattice points in lattice polytopes},
  journal = {Mathematika},
  volume  = {48},
  number  = {1-2},
  pages   = {15--24},
  year    = {2001}
}

@book{AroraBarak,
author = {Arora, S. and Barak, B.},
title = {Computational Complexity: A Modern Approach},
year = {2009}, 
publisher = {Cambridge University Press}, address = {USA}, 
edition = {1st}
}

@InProceedings{Chistikov_FOL,
  author =	{Chistikov, D.},
  title =	{An {I}ntroduction to the {T}heory of {L}inear {I}nteger {A}rithmetic},
  booktitle = {Proc. FSTTCS 2024},
  series    = {LIPIcs},
  volume    = {323},
  pages     = {1:1--1:36},
  year      = {2024},
  publisher = {Schloss Dagstuhl},
}

@article{CodenottiFreyer24,
  title={Lattice reduced and complete convex bodies},
  author={Codenotti, G. and Freyer, A.},
  journal={J. London Math. Soc.},
  volume={110},
  number={4},
  pages={e12982},
  year={2024},
}

@article{soprunov16, 
  title={Eventual quasi-linearity of the {M}inkowski length},
  author={Soprunov, I. and Soprunova, J.},
  journal={Eur. J. Combin.},
  volume={58},
  pages={107--117},
  year={2016},
  publisher={Elsevier}
}

@book{grub07,
  author    = {P. M. Gruber},
  title     = {Convex and {D}iscrete {G}eometry},
  publisher = {Springer},
  address   = {Berlin},
  year      = {2007}
}

@book{zieglerbook,
  title={Lectures on {P}olytopes},
  author={Ziegler, G. M.},
  volume={152},
  year={2012},
  publisher={Springer}
}

@article{BDLM,
  author = {Brandenburg, M.-C. and De Loera, J.~A. and Meroni, C.},
  title = {The {B}est {W}ays to {S}lice a {P}olytope},
  journal = {Math. Comp.},
  volume = {94},
  year = {2025},
  pages = {1003--1042}
}

@article{batyrev1999classification,
  title={On the classification of toric {F}ano 4-folds},
  author={V. V. Batyrev},
  journal={J. Math. Sci.},
  volume={94},
  number={1},
  pages={1021--1050},
  year={1999},
  publisher={Springer}
}

@phdthesis{kasprzyk2006toric,
  title={Toric {F}ano varieties and convex polytopes},
  author={A. Kasprzyk},
  year={2006},
  school={University of Bath}
}

@article{averkov2011maximal,
  title={Maximal lattice-free polyhedra: finiteness and an explicit description in dimension three},
  author={G. Averkov and C. Wagner and R. Weismantel},
  journal={Math. Oper. Res.},
  volume={36},
  number={4},
  pages={721--742},
  year={2011},
  publisher={INFORMS}
}

@article{GritzmannKlee1992,
  author    = {P. Gritzmann and V. Klee},
  title     = {Inner and outer \( j \)-radii of convex bodies in finite-dimensional normed spaces},
  journal   = {Discrete Comput. Geom.},
  volume    = {7},
  number    = {3},
  year      = {1992},
  pages     = {255--280},
  doi       = {10.1007/BF02187841},
  url       = {https://doi.org/10.1007/BF02187841}
}

@inbook{nayar2023extremal,
  title={Extremal sections and projections of certain convex bodies: a survey},
  author={Nayar, P. and Tkocz, T.},
  booktitle={Harmonic analysis and convexity},
  volume={9},
  pages={343--390},
  year={2023},
  publisher={De Gruyter}
}

@inbook{giannopoulos2023inequalities,
  title={Inequalities for sections and projections of convex bodies},
  author={Giannopoulos, A. and Koldobsky, A. and Zvavitch, A.},
  booktitle={Harmonic analysis and convexity},
  volume={9},
  pages={223--256},
  year={2023},
  publisher={De Gruyter}
}

@article{nillziegler2011projecting,
  title={Projecting lattice polytopes without interior lattice points},
  author={Nill, B. and Ziegler, G. M.},
  journal={Math. Oper. Res.},
  volume={36},
  number={3},
  pages={462--467},
  year={2011},
  publisher={INFORMS}
}

@phdthesis{aebThesis,
    author = {Brose, A. E.},
    title = {},
    school = {UC Davis},
    year = {2028},
    note = {(in progress)}
}

@article{perkal1947subdivision,
  title={Sur la subdivision des ensembles en parties de diam{\`e}tre inf{\'e}rieur},
  author={Perkal, J.},
  journal={Colloq. Math.},
  volume={1},
  pages={45},
  year={1947}
}

@article{eggleston1955covering,
  author = {Eggleston, H. G.},
  title = {Covering a three-dimensional set with sets of smaller diameter},
  journal = {J. Lond. Math. Soc.},
  volume = {s1-30},
  number = {1},
  pages = {11--24},
  year = {1955},
}

@inproceedings{Sebo_latticewidth,
author = {Seb\H{o}, A.},
booktitle = {Integer Programming and Combinatorial Optimization (IPCO 1999)},
title = {An {I}ntroduction to {E}mpty {L}attice {S}implices},
year = {1999},
month = {05},
pages = {400-414},
series    = {Lect. Notes Comput. Sci.},
publisher = {Springer},
volume = {1610},
}

@misc{Code_LatticeDiameter,
  author       = {Brose, A.~E.},
  title        = {Computing Lattice Diameter Segments},
  year         = {2026},
  howpublished = {Zenodo},
  note   = {Computer software.
            \url{https://doi.org/10.5281/zenodo.22858560}}
}

@InProceedings{IPCO_ourpaper,
author="Averkov, G.
and Brose, A.
and De Loera, J.
and Lopez-Campos, G.
and Torres, A.J.",
editor="Dey, Santanu S.
and Di Summa, Marco
and Salvagnin, Domenico",
title="On Lattice Diameter Segments: Algorithms and Structure",
booktitle="Integer Programming and Combinatorial Optimization",
year="2026",
publisher="Springer Nature Switzerland",
address="Cham",
pages="220--232",

}


\vfill
\begin{center}
\begin{minipage}[t]{.48\textwidth}\raggedright
\noindent
\textsc{Gennadiy Averkov}\\
\textsc{Brandenburgische Technische Universit\"at Cottbus}\\
\url{averkov@b-tu.de}

\medskip
\noindent
\textsc{Anouk E. Brose}\\
\textsc{University of California, Davis}\\
\url{aebrose@ucdavis.edu}

\medskip
\noindent
\textsc{Jes\'us A. De Loera}\\
\textsc{University of California, Davis}\\
\url{deloera@ucdavis.edu}
\end{minipage}\hfill
\begin{minipage}[t]{.48\textwidth}\raggedright
\noindent
\textsc{Gyivan Lopez-Campos}\\
\textsc{University of California, Davis}\\
\url{glopezcampos@ucdavis.edu}

\medskip
\noindent
\textsc{Antonio J. Torres}\\
\textsc{Smith College, Northampton, MA}\\
\url{atorres@smith.edu}
\end{minipage}
\end{center}

\end{document}